\documentclass[a4paper,11pt]{article}
\usepackage{amsmath,amsthm,amssymb}
\usepackage{enumerate}

\usepackage[utf8]{inputenc}
\usepackage{amsfonts}

\usepackage[english]{babel}
\usepackage[utf8]{inputenc}
\usepackage[a4paper]{geometry}
\usepackage{latexsym}
\usepackage{amscd}
\usepackage{graphics}
\usepackage{color}
\usepackage{array}
\usepackage{mathrsfs}
\usepackage{graphicx}
\usepackage{stmaryrd}
\usepackage{mathabx}
\usepackage{dsfont}
\usepackage{comment}
\usepackage{tikz}
\usepackage{tikz-cd}
\usepackage{mathtools}
\usepackage{ulem}

\def\subjclass#1{{\renewcommand{\thefootnote}{}%
\footnote{\emph{Mathematics Subject Classification (2020):} #1}}}

\def\keywords#1{{\renewcommand{\thefootnote}{}%
\footnote{\emph{Keywords:} #1}}}

\def\ackn#1{{\renewcommand{\thefootnote}{}%
\footnote{#1}}}

\newtheorem{thm}{Theorem}[section]
\newtheorem{cor}[thm]{Corollary}
\newtheorem{lem}[thm]{Lemma}

\newtheorem{prop}[thm]{Proposition}
\newtheorem{conj}[thm]{Conjecture}

\newtheorem{question}[thm]{Question}
\newtheorem{claim}{Claim}[thm]

\theoremstyle{definition}
\newtheorem{defin}[thm]{Definition}

\numberwithin{equation}{section}

\newcommand{\N}{\mathbb{N}}

\newcommand{\R}{\mathbb{R}}

\makeatletter
\newcommand{\alaligne}{~\vspace*{\topsep}\nobreak\@afterheading}
\makeatother

\begin{document}

\title{The ergodicity of Orlicz sequence spaces}

\author{No\'e de Rancourt \and Ond\v{r}ej Kurka}

\date{}

\maketitle

\begin{abstract}
We prove that non-Hilbertian separable Orlicz sequence spaces are ergodic, i.e., the equivalence relation $\mathbb{E}_0$ Borel reduces to the isomorphism relation between subspaces of every such space. This is done by exhibiting non-Hilbertian asymptotically Hilbertian subspaces in those spaces, and appealing to a result by Anisca. In particular, each non-Hilbertian Orlicz sequence space contains continuum many pairwise non-isomorphic subspaces.

As a consequence, we prove that the twisted Hilbert spaces $\ell_2(\phi)$ constructed by Kalton and Peck are either Hilbertian, or ergodic. This applies in particular to the Kalton--Peck space $Z_2$ and all twisted Hilbert spaces generated by complex interpolation between Orlicz sequence spaces.

\end{abstract}


\ackn{N. de Rancourt acknowledges support from the Labex CEMPI (ANR-11-LABX-0007-01) and the CDP C2EMPI, together with the French State under the France-2030 programme, the University of Lille, the Initiative of Excellence of the University of Lille, the European Metropolis of Lille for their funding and support of the R-CDP-24-004-C2EMPI project. O.~Kurka was supported by the Czech Science Foundation, project no.~GA\v{C}R~22-07833K, and by the Academy of Sciences of the Czech Republic (RVO 67985840).}

\setcounter{footnote}{0}

\section{Introduction}

A celebrated result by Komorowski and Tomczak-Jaegermann \cite{ktj} and Gowers \cite{GowersRamsey}, solving Banach's homogeneous space problem, asserts that if a Banach space is isomorphic to all of its infinite-dimensional subspaces, then it must be Hilbertian (where \textit{subspace} means closed vector subspace, and \textit{Hilbertian} means isomorphic to a Hilbert space). It is therefore natural to ask how many infinite-dimensional subspaces a non-Hilbertian separable Banach space can have, up to isomorphism. This question, originally asked by Godefroy, can be made more precise by the classification theory for definable equivalence relations, a branch of descriptive set theory aimed at quantifying the complexity of classification problems in mathematics. This theory classifies definable (in most cases, Borel or analytic) equivalence relations on Polish spaces via the \textit{Borel reducibility quasiorder}, defined by $E \leqslant_B F$ (read ``$E$ Borel reduces to $F$'') iff there is a Borel mapping $ f $ from the domain of $ E $ to the domain of $ F $ such that $ f(x) \mathrel{F} f(y) \Leftrightarrow x\mathrel{E}y $. This quasiorder gives rise to a hierarchy of complexity classes; the complexity class of an equivalence relation gives strictly more information than the number of its equivalence classes. The interested reader can find more information, including precise definitions and motivations, in the book \cite{kanovei}.

An important result by Harrington, Kechris and Louveau \cite{hkl} asserts that there exists a $\leqslant_B$-least Borel equivalence relation above the equality relation on $\R$, namely the equivalence relation $\mathbb{E}_0$ on the Cantor space $2^\N$ defined by $a \mathrel{\mathbb{E}_0} b$ iff the sequences $a$ and $b$ eventually agree. We below recall the definition of ergodic Banach spaces, introduced by Ferenczi and Rosendal in \cite{FerencziRosendalErgodic} (where the reader can find a more formal definition).

\begin{defin}[Ferenczi--Rosendal 2005]
    A separable Banach space $X$ is \textit{ergodic} if $\mathbb{E}_0$ Borel reduces to the isomorphism relation between subspaces of $X$.
\end{defin}

In particular, an ergodic Banach space should have continuum-many pairwise non-isomorphic subspaces, and those subspaces cannot be classified, up to isomorphism, with real numbers as invariants. Let us mention that it easily follows from the definition that a subspace of a non-ergodic space is itself non-ergodic; this fact will often be implicitly used. The following conjecture, generalizing the solution of the homogeneous space problem, was stated in \cite{FerencziRosendalErgodic}.

\begin{conj}[Ferenczi--Rosendal 2005]\label{ConjErgodic}
    Every non-Hilbertian separable Banach space is ergodic.
\end{conj}

This conjecture is still widely open, although much progress has been done, see for instance \cite{FerencziRosendalErgodic, RosendalIncomparable, FerencziMinimal, Anisca, CuellarNearHilbert, CuellarRancourtFerenczi}. A good summary of early results on the conjecture is contained in the survey \cite{FerencziRosendalSurvey}.

A particularly striking result, proved in \cite{CuellarNearHilbert}, is the following. Recall that a Banach space is said to be \textit{near Hilbert} if it has type $p$ and cotype $q$ for all $1 \leqslant p < 2 < q \leqslant \infty$.

\begin{thm}[Cuellar Carrera 2018]\label{Cuellar}
    Non-ergodic separable Banach spaces are near Hilbert.
\end{thm}

It follows in particular that the spaces $c_0$ and $\ell_p$ for $1 \leqslant p < \infty, p \neq 2,$ are ergodic (in the case of $c_0$ and $\ell_p$ for $1 \leqslant p < 2$, this had previously been proved by Ferenczi and Galego \cite{FerencziGalego}). By Theorem \ref{Cuellar}, in order to prove or disprove Conjecture \ref{ConjErgodic}, what remains to be explored is the realm of non-Hilbertian separable near Hilbert spaces. Examples of such spaces include:
\begin{itemize}
    \item some Orlicz sequence spaces $h_M$ (i.e. spaces generalizing the $\ell_p$ spaces, see Section \ref{sec:Def} for a definition), for instance those associated to the functions $M(t) = t^2|\log t|^\alpha$, $\alpha \in \R \setminus \{ 0 \}$;
    \item all non-Hilbertian separable \textit{twisted Hilbert spaces}, i.e. Banach spaces $X$ having a subspace $Y$ such that both $Y$ and $X/Y$ are Hilbertian (see \cite[3.11.4]{CabelloCastillo} for the proof that twisted Hilbert spaces are near Hilbert).
\end{itemize}
In particular, the question of whether the famous Kalton--Peck space $Z_2$, a classical example of a twisted Hilbert space, is ergodic, has been around for some years.

The main result of our paper is the following.

\begin{thm}\label{MainMain}
    For every Orlicz function $M$, the Orlicz sequence space $h_M$ is either Hilbertian, or ergodic.
\end{thm}

Note that the ergodicity of some near Hilbert Orlicz sequence spaces, namely spaces $h_M$ for $M(t) = t^2|\log t|^\alpha$ with $\alpha > 2$, had already been proved by Cuellar Carrera (unpublished) before the beginning of our research. To do so, he extended his method for proving Theorem \ref{Cuellar}, which consists in using Szankowski's techniques from \cite{Szankowski} to build sufficiently many pairwise non-isomorphic subspaces failing the approximation property. Cuellar Carrera recently informed us that Torres Guzm\'an and himself managed to extend this proof to the case $\alpha < -2$, and will soon release it in a preprint, along with related results. However, their method cannot be extended to all separable Orlicz sequences spaces since, as was proved by Johnson and Szankowski \cite{JohnsonSzankowski}, some of those spaces fail to have a subspace without the approximation property.

Our method for proving Theorem \ref{MainMain} is completely different. While the case of non-near Hilbert spaces follows from Theorem \ref{Cuellar}, in the near Hilbert case, we build an asymptotically Hilbertian, non-Hilbertian block subspace of $h_M$ (see Section \ref{sec:Def} for a definition of an asymptotically Hilbertian space) and appeal to the following result by Anisca \cite{Anisca}.

\begin{thm}[Anisca]\label{thm:Anisca}
    Every non-Hilbertian, asymptotically Hilbertian separable Banach space is ergodic.
\end{thm}

Note that, to build a Borel reduction of $\mathbb{E}_0$ to the isomorphism relation between subspaces of a non-Hilbertian, asymptotically Hilbertian separable Banach space $X$, Anisca only uses subspaces having an FDD $(F_n)$ (where, if $X$ comes with a basis, the $F_n$'s can even be chosen to have finite and successive supports). Thus, unlike Cuellar Carrera, to prove Theorem \ref{MainMain} in the case of near Hilbert spaces, we only use subspaces with the approximation property.

It would be interesting to know whether our methods can be extended to prove the ergodicity of more general classes of spaces than Orlicz sequences spaces, for instance spaces with a symmetric basis,
or Musielak--Orlicz sequence spaces (see Section \ref{sec:PartCase} for a definition). For this, one would need an analogue of Theorem \ref{thm:main} for those spaces; we spell out a possible statement as a question below.

\begin{question}\label{quest:MusielakOrSymmetricContainAsympHilbert}
    Let $X$ be either a Banach space with a symmetric basis, or a Musielak--Orlicz sequence space.
    Suppose that $X$ is non-Hilbertian and near Hilbert. Does $X$ necessarily contain a non-Hilbertian, asymptotically Hilbertian subspace? 
\end{question}

If the answer to Question \ref{quest:MusielakOrSymmetricContainAsympHilbert} happens to be negative, a way to prove ergodicity of an $X$ as above would be to find a non-Hilbertian subspace $Y$ of $X$ that does not contain any $d_2$-minimal subspaces; this condition is weaker than $Y$ being asymptotically Hilbertian, by \cite[Corollary 5.24]{CuellarRancourtFerenczi}. Recall that the notion of $d_2$-minimal spaces has been defined in \cite{CuellarRancourtFerenczi}, where it has been proved that separable, non-Hilbertian Banach spaces without $d_2$-minimal subspaces must be ergodic (Theorem 6.5). The peculiar local properties of $d_2$-minimal Banach spaces (see \cite[Lemma 5.22]{CuellarRancourtFerenczi}) make us think that methods similar to those developed in Sections \ref{sec:PartCase} and \ref{sec:ProofMain} of the present paper could still be helpful to answer such a weaker version of Question \ref{quest:MusielakOrSymmetricContainAsympHilbert}.

From Theorem \ref{MainMain}, we draw some consequences about twisted Hilbert spaces. An important part of their theory was developed by Kalton and Peck in the seminal paper \cite{KaltonPeck}, where a family of examples, the spaces $\ell_2(\phi)$, are built (the definition of spaces $\ell_2(\phi)$ will be recalled in Section \ref{sec:Twisted}; here, $\phi \colon \R \to \R$ is a Lipschitz function). We prove the following result.

\begin{thm}\label{ErgTwisted}
    For every Lipschitz function $\phi \colon \R \to \R$, the space $\ell_2(\phi)$ is either Hilbertian, or ergodic.
\end{thm}

In the special case when $\phi(t) = t$, the space $\ell_2(\phi)$ is the above mentioned Kalton--Peck space. The problem of its ergodicity is immediately solved by Theorem~\ref{ErgTwisted}.

\begin{thm}
    The Kalton--Peck space $Z_2$ ergodic.
\end{thm}

An interesting way to produce twisted Hilbert spaces, developed in \cite{CastilloFerencziGonzalez}, is through the use of complex interpolation. As will be explained in Section \ref{sec:Twisted}, the family of spaces $\ell_2(\phi)$ contains all twisted Hilbert spaces generated by complex interpolation between Orlicz sequence spaces, so those spaces are ergodic by Theorem \ref{ErgTwisted}. Finally, let us mention that the following question is still open.

\begin{question}\label{quest:ErgodicityTwisted}
    Is every non-Hilbertian separable twisted Hilbert space ergodic?
\end{question}

This paper is organized as follows. In Section \ref{sec:Def}, most basic definitions and known results used in the paper, in particular those about Orlicz sequence spaces, will be introduced. In Section \ref{sec:PartCase}, our method will be introduced by proving a particular case of Theorem \ref{MainMain}, namely this of spaces $h_M$ where for all $t \in (0, 1]$, $\lim_{\lambda \to 0^+} M(\lambda t)/M(\lambda) = t^2$. This condition is in particular satisfied when $M(t) = t^2|\log t|^\alpha$, $\alpha \in \R$. In Section \ref{sec:ProofMain}, Theorem \ref{MainMain} is proved in full generality. Finally, in Section \ref{sec:Twisted}, consequences of Theorem \ref{MainMain} for twisted Hilbert spaces are obtained.

\section{Notation, definitions and known results}\label{sec:Def}

In this work, most of the results are valid in both real and complex settings. By $ \mathbb{K} $ we denote the scalar field, where we allow both possibilities $ \mathbb{K} = \mathbb{R} $ and $ \mathbb{K} = \mathbb{C} $. Moreover, we define $ \mathbb{R}_{+} = [0, \infty) $.

The unit sphere of a Banach space $X$ will be denoted by $S_X$. Unless otherwise specified, the norms on all Banach spaces will be denoted by $\|\cdot\|$; this includes operator norms. If $E$ and $F$ are two Banach spaces of the same finite dimension, we define the \textit{Banach--Mazur distance} between $E$ and $F$ as
$$d_{BM}(E, F) = \inf\left\{\|T\|\cdot \|T^{-1}\| \,\left|\, T \colon E \to F \text{ is an isomorphism}\right.\right\}.$$

For the basic definitions of Orlicz functions and Orlicz sequence spaces, we follow \cite{LindenstraussTzafririI}. In particular, an \textit{Orlicz function} is a convex function $M \colon \mathbb{R}_{+} \to \mathbb{R}_{+}$ satisfying $M(0) = 0$ and $\lim_{t \to + \infty}M(t) = + \infty$ (it follows that $M$ is continuous and nondecreasing). If $M$ is an Orlicz function, then for each $x \colon \mathbb{N} \to \mathbb{K}$, we define
$$\|x\|_M := \inf\left\{\rho > 0 \,\left|\, \sum_{n=1}^{\infty} M\left(\frac{|x(n)|}{\rho}\right) \leqslant 1\right.\right\} .$$
We denote by $\ell_M$ the vector space of all $x \colon \mathbb{N} \to \mathbb{K}$ for which $\|x\|_M < + \infty$. The space $\ell_M$ endowed with the norm $\|\cdot \|_M$ is a Banach space, called the \textit{Orlicz sequence space} associated to $M$.

For $n \in \N$, we let $e_n \in \ell_M$ be the vector such that $e_n(n) = 1$ and $e_n(i) = 0$ for $i \neq n$. We denote by $h_M$ the closed linear span of the family $(e_n)_{n \in \N}$. As shown in \cite[Proposition 4.a.2]{LindenstraussTzafririI}, the space $h_M$ can be characterized equivalently as the set of vectors $x \in \ell_M$ such that
$$\sum_{n=1}^{\infty} M\left(\frac{|x(n)|}{\rho}\right) < + \infty$$
for all $\rho > 0$. The family $(e_n)_{n \in \N}$ forms a basis of the space $h_M$, which is symmetric with constant $1$. We will refer to it as the \textit{canonical basis} of $h_M$ in what follows.

The following result is part of \cite[Proposition 4.a.5]{LindenstraussTzafririI}.

\begin{prop}\label{Prop:IsomOrlicz}
Let $M$ and $N$ be two Orlicz functions. The following are equivalent:
\begin{enumerate}[(1)]
    \item there exist constants $C, K > 0$ and $t_0>0$ such that for $0 \leq t \leq t_0$, we have:
    $$C^{-1}M(K^{-1}t) \leqslant N(t) \leqslant CM(Kt);$$
    \item the canonical bases of $h_M$ and $h_N$ are equivalent.
\end{enumerate}
\end{prop}

Occasionally, the notation $h_M $ can be used for a nonnegative function $M$ with $M(0) = 0$ that is only defined and convex on some interval $ [0, \varepsilon) $,  $ \varepsilon > 0 $. This is for instance the case when $M(t) = t^2|\log t|^\alpha$, $\alpha \in \R \setminus \{0\}$. In such a case, we can find an Orlicz function $\widetilde{M}$ that coincides with $M$ on a smaller interval $ [0, \delta) $. By Proposition \ref{Prop:IsomOrlicz}, the space $h_{\widetilde{M}}$ as a set, and the equivalence class of its norm, do not depend on the choice of $\widetilde{M}$ satisfying the above property. We will hence denote by $h_M$ the space $h_{\widetilde{M}}$ whose norm is only defined up to equivalence.

Next, we recall an important result for proving Theorem~\ref{MainMain}, namely \cite[Theorem~4.a.9]{LindenstraussTzafririI}, characterizing when an Orlicz sequence space contains $ \ell_{p} $, resp.~$ c_{0} $. For an Orlicz function $ M $, we define
$$ \alpha_{M} = \sup \Big\{ q : \sup_{\lambda, t \in (0, 1]} \frac{M(\lambda t)}{M(\lambda) t^{q}} < \infty \Big\}, $$
$$ \beta_{M} = \inf \Big\{ q : \inf_{\lambda, t \in (0, 1]} \frac{M(\lambda t)}{M(\lambda) t^{q}} > 0 \Big\}. $$
Observe that we always have $\alpha_M \leqslant \beta_M$.

\begin{thm}[\cite{LindenstraussTzafririI}] \label{thm:containmentofellp}
The space $ \ell_{p} $, or $ c_{0} $ if $ p = \infty $, is isomorphic to a subspace of an Orlicz sequence space $ h_{M} $ if and only if $ p \in [\alpha_{M}, \beta_{M}] $.
\end{thm}

Before going further, two remarks are in order. First, by \cite[Proposition 4.a.4]{LindenstraussTzafririI}, we have $\ell_M = h_M$ iff $\ell_M$ is separable, iff $M$ satisfies the \textit{$\Delta_2$-condition at $0$}, namely $\limsup_{t\to 0+}M(2t)/M(t)  < \infty$. If this condition does not hold, then it is easily seen that we must have $\beta_M = \infty$, hence $c_0$ embeds into $h_M$. Therefore, for all spaces we are actually interested in, we will have $\ell_M = h_M$; in such a case, it is much more common in the literature to denote this space by $\ell_M$. However, in the present paper, we chose to keep the notation $h_M$ in order to be able to state our results in full generality.

Second, if $\alpha_M = \beta_M = 2$, then by the Remark at the bottom of page 140 in \cite{LindenstraussTzafririII}, $h_M$ satisfies an upper $p$-estimate and a lower $q$-estimate for all $1 \leqslant p < 2 < q \leqslant \infty$. It follows from \cite[Theorem 1.c.16]{LindenstraussTzafririII} that $h_M$ is near Hilbert. Conversely, if $h_M$ is near Hilbert, then it cannot contain a copy of any $\ell_p$, $p \neq 2$, nor of $c_0$, therefore by Theorem \ref{thm:containmentofellp}, one has $\alpha_M = \beta_M = 2$. Although this equivalence will not be formally used in the proof of Theorem \ref{MainMain}, it is good to have it in mind. This is for instance how one can prove that the space $h_M$, where $M(t) = t^2|\log t|^\alpha$, is near Hilbert for all $\alpha \in \R$.

We now recall the definition of asymptotically Hilbertian spaces and its quantitative versions.

\begin{defin}
A Banach space $X$ is said to be:
\begin{itemize}
    \item \textit{$K$-asymptotically Hilbertian}, for $K \geqslant 1$, if for every $n \in \N$, there exists a finite-codimensional subspace $Y \subseteq X$ such that for every $n$-dimensional vector subspace $E \subseteq Y$, we have $d_{BM}(E, \ell_2^n) \leqslant K$;
    \item \textit{$K^+$-asymptotically Hilbertian}, for $K \geqslant 1$, if it is $(K+\varepsilon)$-asymptotically Hilbertian for every $\varepsilon > 0$;
    \item \textit{asymptotically Hilbertian} if it is $K$-asymptotically Hilbertian for some $K \geqslant 1$.
\end{itemize}
\end{defin}

We can now state the main result that will be actually proved in this paper, and from which Theorem \ref{MainMain} will follow.

\begin{thm}\label{thm:main}
    Let $ M $ be an Orlicz function such that $ \alpha_{M} = \beta_{M} = 2 $. If $ h_{M} $ is not Hilbertian, then it contains a non-Hilbertian $ 1^{+} $-asymptotically Hilbertian subspace.
\end{thm}

\begin{proof}[Proof of Theorem \ref{MainMain}]
Let $M$ be an Orlicz function such that $h_M$ is non-Hilbertian. If $h_M$ contains an isomorphic copy of $c_0$ or an $\ell_p$, $p \neq 2$, then it is ergodic by Theorem \ref{Cuellar}. Otherwise, by Theorem \ref{thm:containmentofellp}, we must have $\alpha_M = \beta_M = 2$, so $h_M$ contains an asymptotically Hilbertian, non-Hilbertian subspace, so it is ergodic by Theorem \ref{thm:Anisca}.    
\end{proof}

\section{A special case}\label{sec:PartCase}

The aim of this section is to prove the following proposition, which is a special case of Theorem \ref{thm:main}.

\begin{prop}[the special case] \label{prop:special}
Let $ M $ be an Orlicz function such that
$$ \lim_{\lambda \to 0+} \frac{M(\lambda t)}{M(\lambda)} = t^{2} $$
for every $ t \in (0, 1] $. If $ h_{M} $ is not Hilbertian, then it contains a non-Hilbertian $ 1^{+} $-asymptotically Hilbertian subspace.
\end{prop}

 Although neither Proposition \ref{prop:special} nor its proof at the end of this section are used for proving the general result, the proof may be helpful for understanding and shows what is possible to obtain by considering a block subspace generated by vectors that are constant on their supports. Moreover, Proposition \ref{prop:special} is sufficient for proving the ergodicity of spaces $h_M$ for $M(t) = t^2|\log t|^\alpha$.

Let $ M_{1}, M_{2}, \dots $ be a sequence of Orlicz functions. For $ x : \mathbb{N} \to \mathbb{K} $, we define
$$ \Vert x \Vert_{(M_{i})_{i}} = \inf \Big\{ \rho > 0 : \sum_{i=1}^{\infty} M_{i} \big( \tfrac{|x(i)|}{\rho} \big) \leq 1 \Big\}. $$
We denote by $ \ell_{(M_{i})_{i}} $ the vector space of all sequences $ x $ such that $ \Vert x \Vert_{(M_{i})_{i}} < \infty $. The space $ \ell_{(M_{i})_{i}} $ endowed with the norm $ \Vert \cdot \Vert_{(M_{i})_{i}} $ is called the \textit{Musielak--Orlicz sequence space} associated to $ (M_{i})_{i} $.
Moreover, by $ h_{(M_{i})_{i}} $ we denote the subspace of $ \ell_{(M_{i})_{i}} $ generated by the canonical basic sequence $ (e_{i})_{i} $.

Our purpose of this notation is the following description of block subspaces of $ \ell_{M} $, cf.~with \cite[p.~141]{LindenstraussTzafririI}. If $ x_{1}, x_{2}, \dots $ is a sequence in $ \ell_{M} $ with non-empty finite and pairwise disjoint supports, then it is $ 1 $-equivalent with the canonical basis $ e_{i} $ in $ h_{(M_{i})_{i}} $, where $ M_{i} $ is defined by
$$ M_{i}(t) = \sum_{n=1}^{\infty} M(|x_{i}(n)|t), \quad t \geq 0. $$

\begin{lem} \label{lem:DistToEll2}
Let $ C \geq 1 $ and $ 0 < \nu < 1 $. Let $ M_{1}, M_{2}, \dots $ be Orlicz functions with $ M_{i}(1) \geq 1 $ such that
$$ \frac{1}{C} t^{2} \leq \frac{M_{i}(\lambda t)}{M_{i}(\lambda)} \leq C t^{2}, \quad 0 < \lambda \leq 1, \; \nu \leq t \leq 1. $$
Then every subspace $ E \subseteq \ell_{(M_{i})_{i}} $ of dimension $ n \leq 1/\nu $ satisfies
$$ d_{BM}(E, \ell_{2}^{n}) \leq \sqrt{C(C+\nu^{2})/(1-\nu)}. $$
\end{lem}

\begin{claim} \label{claim:DistToEll2}
Let $ M_{1}, M_{2}, \dots $ be as in Lemma~\ref{lem:DistToEll2}. Then every $ x \in \ell_{(M_{i})_{i}} $ with $ \Vert x \Vert = 1 $ satisfies $ |x(i)| \leq 1 $ for each $ i $ and
$$ \sum_{i=1}^{\infty} M_{i}(|x(i)|) = 1. $$
Moreover, we have $ h_{(M_{i})_{i}} = \ell_{(M_{i})_{i}} $.
\end{claim}

\begin{proof}
From the definition of $ \Vert x \Vert $, we get that $ \sum M_{i}(|x(i)|/\rho) \leq 1 $ for $ \rho > 1 $. Clearly, $ \lim_{\rho \to 1} M_{i}(|x(i)|/\rho) \to M_{i}(|x(i)|) $ for each $ i $. By the monotone convergence theorem, we deduce that
$$ \sum M_{i}(|x(i)|) \leq 1. $$
In particular, for each $ i $, since $ M_{i}(|x(i)|) \leq 1 \leq M_{i}(1) $, we get $ |x(i)| \leq 1 $.

By our assumption, for each $ i $ with $ |x(i)| \geq \nu $, we have $ M_{i}(|x(i)|) \geq \frac{1}{C} |x(i)|^{2} \cdot M_{i}(1) \geq \frac{1}{C} |x(i)|^{2} $. So, $ \sum_{|x(i)| \geq \nu} |x(i)|^{2} \leq C < \infty $, and $ |x(i)| \geq \nu $ only for finitely many indices $ i $. If $ |x(i)| < \nu $, then $ M_{i}(|x(i)|) \geq \frac{1}{C} \nu^{2} M_{i}(|x(i)|/\nu) $. So, $ \sum_{|x(i)| < \nu} M_{i}(|x(i)|/\nu) \leq C/\nu^{2} $. We obtain
$$ \sum M_{i}(|x(i)|/\nu) < \infty. $$
Using the definition of $ \Vert x \Vert $ again, we get that $ \sum M_{i}(|x(i)|/\rho) > 1 $ for $ 0 < \rho < 1 $. By the dominated convergence theorem,
$$ \sum M_{i}(|x(i)|) \geq 1. $$

Concerning the moreover part, note that there is $ i_{1} $ such that $ \sum_{i=i_{1}}^{\infty} M_{i}(|x(i)|/\nu) \leq 1 $, and so $ \Vert P_{i_{1}} x \Vert \leq \nu = \nu \Vert x \Vert $, where $ P_{i_{1}} $ denotes the projection on the space of vectors supported by $ \{ i_{1}, i_{1}+1, \dots \} $. By induction, there are $ i_{k} \in \mathbb{N} $ such that $ \Vert P_{i_{k}} x \Vert \leq \nu^{k} \Vert x \Vert $. Consequently, $ x \in h_{(M_{i})_{i}} $.
\end{proof}

\begin{proof}[Proof of Lemma \ref{lem:DistToEll2}]
Fix an Auerbach basis $ (x_{1}, \dots, x_{n}) $ of $ E $ and let
$$ x = \frac{|x_{1}| + \dots + |x_{n}|}{\||x_{1}| + \dots + |x_{n}|\|}. $$
We define
$$ \vvvert y \vvvert^{2} = \sum_{x(i) > 0} |y(i)|^{2} \frac{M_{i}(x(i))}{x(i)^{2}}, \quad y \in E. $$
The lemma will be proved once we show that
$$ y \in S_{E} \quad \Rightarrow \quad \frac{1}{C}(1 - \nu) \leq \vvvert y \vvvert^{2} \leq C + \nu^{2}, $$
since $ (E, \vvvert \cdot \vvvert) $ is isometric to $ \ell_{2}^{n} $.

So, let $ y \in S_{E} $ be arbitrary. Write $ y = a_{1}x_{1} + \dots + a_{n}x_{n} $ with $ a_{1}, \dots, a_{n} \in \mathbb{K} $. Since the basis $ (x_{1}, \dots, x_{n}) $ is Auerbach, we have $ |a_{k}| \leq 1 $ for $ 1 \leq k \leq n $. Thus
$$
\begin{aligned}
|y| & \leq |x_{1}| + \dots + |x_{n}| \\
 & = \| |x_{1}| + \dots + |x_{n}| \| \cdot x \\
 & \leq nx.
\end{aligned}
$$
By Claim~\ref{claim:DistToEll2}, we have $ x(i) \leq 1, |y(i)| \leq 1 $ and
$$ \sum_{i=1}^{\infty} M_{i}(x(i)) = 1, \quad \sum_{i=1}^{\infty} M_{i}(|y(i)|) = 1. $$
Let us realize that
$$ |y(i)| \geq \nu x(i) > 0 \quad \Rightarrow \quad \frac{1}{C} \frac{|y(i)|^{2}}{x(i)^{2}} \leq \frac{M_{i}(|y(i)|)}{M_{i}(x(i))} \leq C \frac{|y(i)|^{2}}{x(i)^{2}}. $$
If $ |y(i)| \leq x(i) $, we apply the assumption of the lemma on $ \lambda = x(i), t = |y(i)|/x(i) $, otherwise we consider $ \lambda = |y(i)|, t = x(i)/|y(i)| $, where we use $ |y| \leq nx \leq (1/\nu)x $.

It follows that
$$ \frac{1}{C} \sum_{|y| \geq \nu x > 0} M_{i}(|y(i)|) \leq \sum_{|y| \geq \nu x > 0} |y(i)|^{2} \frac{M_{i}(x(i))}{x(i)^{2}} \leq C \sum_{|y| \geq \nu x > 0} M_{i}(|y(i)|). $$
We obtain
$$
\begin{aligned}
\vvvert y \vvvert^{2} & = \sum_{|y| \geq \nu x > 0} |y(i)|^{2} \frac{M_{i}(x(i))}{x(i)^{2}} + \sum_{|y| < \nu x} |y(i)|^{2} \frac{M_{i}(x(i))}{x(i)^{2}} \\
 & \leq C \sum_{|y| \geq \nu x > 0} M_{i}(|y(i)|) + \nu^{2} \sum_{|y| < \nu x} M_{i}(x(i)) \\
 & \leq C + \nu^{2}.
\end{aligned}
$$
At the same time, since
$$ \sum_{|y| < \nu x} M_{i}(|y(i)|) \leq \sum_{|y| < \nu x} M_{i}(\nu x(i)) \leq \nu \sum_{|y| < \nu x} M_{i}(x(i)) \leq \nu, $$
we have
$$
\begin{aligned}
\vvvert y \vvvert^{2} & \geq \sum_{|y| \geq \nu x > 0} |y(i)|^{2} \frac{M_{i}(x(i))}{x(i)^{2}} \\
 & \geq \frac{1}{C} \sum_{|y| \geq \nu x > 0} M_{i}(|y(i)|) = \frac{1}{C} \Big( 1 - \sum_{|y| < \nu x} M_{i}(|y(i)|) \Big) \\
 & \geq \frac{1}{C}(1 - \nu),
\end{aligned}
$$
which completes the proof of the lemma.
\end{proof}

\begin{lem} \label{lem:AsympHilb}
Let $ M_{1}, M_{2}, \dots $ be Orlicz functions with $ M_{i}(1) \geq 1 $ such that for every $ \varepsilon > 0 $ and $ 0 < \nu < 1 $, there is $ i_{0} $ such that
$$ \frac{1}{1+\varepsilon} t^{2} \leq \frac{M_{i}(\lambda t)}{M_{i}(\lambda)} \leq (1+\varepsilon) t^{2}, \quad i \geq i_{0}, \; 0 < \lambda \leq 1, \; \nu \leq t \leq 1. $$
Then the space $ \ell_{(M_{i})_{i}} $ is $ 1^{+} $-asymptotically Hilbertian.
\end{lem}

\begin{proof}
Let $ n \in \mathbb{N} $ and $ K > 1 $. We choose $ \varepsilon > 0 $ and $ 0 < \nu < 1 $ such that $ \nu \leq 1/n $ and $ \sqrt{(1+\varepsilon)(1+\varepsilon+\nu^{2})/(1-\nu)} \leq K $. Let $ i_{0} $ be as in the assumption. Then we can apply Lemma~\ref{lem:DistToEll2} on the sequence $ (M_{i})_{i \geq i_{0}} $, concluding that subspace consisting of all sequences in $ \ell_{(M_{i})_{i}} $ supported by $ \{ i_{0}, i_{0}+1, \dots \} $ has the property that any subspace $ E $ of dimension $ n $ satisfies
$$ d_{BM}(E, \ell_{2}^{n}) \leq \sqrt{(1+\varepsilon)(1+\varepsilon+\nu^{2})/(1-\nu)} \leq K. $$
This shows that $ \ell_{(M_{i})_{i}} $ is indeed $ 1^{+} $-asymptotically Hilbertian.
\end{proof}

\begin{proof}[Proof of Proposition \ref{prop:special}]
For $ k \in \mathbb{N} $, let us denote $ \lambda_{k} = M^{-1}(1/k) $ and let
$$ M_{k}(t) = \frac{M(\lambda_{k}t)}{M(\lambda_{k})} = kM(\lambda_{k}t), \quad t \geq 0. $$
By Proposition~\ref{Prop:IsomOrlicz}, the spaces $ h_{M} $ and $ h_{M_{k}} $ are isomorphic, thus $ h_{M_{k}} $ is not Hilbertian. For a large enough $ s(k) \in \mathbb{N} $, the Banach-Mazur distance of the subspace $ \mathrm{span} \{ e_{1}, \dots, e_{s(k)} \} $ of $ h_{M_{k}} $ to the $ s(k) $-dimensional Euclidean space is at least $ k $.

Let $ (k_{i})_{i} $ be the sequence of the form $ 1, \dots, 1, 2, \dots, 2, 3, \dots, 3, \dots $, in which $ k $ appears $ s(k) $ times. We can find a block sequence $ x_{1}, x_{2}, \dots $ in $ h_{M} $ such that $ \mathrm{supp} \, x_{i} $ has exactly $ k_{i} $ elements and $ x_{i}(n) = \lambda_{k_{i}} $ for $ n \in \mathrm{supp} \, x_{i} $. Then the space generated by $ (x_{i})_{i} $ is not Hilbertian. Indeed, for each $ k $, the vectors $ \{ x_{i} : k_{i} = k \} $ generate a subspace isometric to the subspace $ \mathrm{span} \{ e_{1}, \dots, e_{s(k)} \} $ of $ h_{M_{k}} $, whose distance to $ \ell_{2}^{s(k)} $ is at least $ k $. It remains to show that the space generated by $ (x_{i})_{i} $ is $ 1^{+} $-asymptotically Hilbertian. Note that this subspace is isometric to the space $ h_{(N_{i})_{i}} $, where $ N_{i} = M_{k_{i}} $.

Since the functions $ M(\lambda t)/M(\lambda) $ are non-decreasing, they actually converge uniformly to $ t^{2} $ on $ [0, 1] $ as $ \lambda \to 0+ $. It follows that for every $ 0 < \nu < 1 $ and $ \varepsilon > 0 $, there is $ \delta > 0 $ such that
$$ \frac{1}{1+\varepsilon} t^{2} \leq \frac{M(\lambda t)}{M(\lambda)} \leq (1+\varepsilon) t^{2}, \quad 0 < \lambda \leq \delta, \; \nu \leq t \leq 1. $$
If we choose $ k(\delta) \in \mathbb{N} $ such that $ \lambda_{k(\delta)} \leq \delta $, then
$$ \frac{1}{1+\varepsilon} t^{2} \leq \frac{M_{k}(\lambda t)}{M_{k}(\lambda)} \leq (1+\varepsilon) t^{2}, \quad k \geq k(\delta), \; 0 < \lambda \leq 1, \; \nu \leq t \leq 1, $$
as $ \frac{M_{k}(\lambda t)}{M_{k}(\lambda)} = \frac{M(\lambda_{k} \lambda t)}{M(\lambda_{k} \lambda)} $ and $ 0 < \lambda_{k} \lambda \leq \lambda_{k(\delta)} \leq \delta $. It follows that the sequence $ N_{i} = M_{k_{i}} $ satisfies the assumptions of Lemma~\ref{lem:AsympHilb}, which completes the proof.
\end{proof}

\section{The general case}\label{sec:ProofMain}

The goal of this section is to prove Theorem \ref{thm:main} in its full generality.

\begin{lem} \label{lem:GenCaseStep1}
Let $ M $ be an Orlicz function such that $ \alpha_{M} = \beta_{M} = 2 $. Let $ 0 < \tau < 1 $, $ 0 < \kappa < 1 $ and $ \eta > 0 $ be such that $ \kappa > \frac{1}{1+\eta} $. Then there is $ R \in \mathbb{N} $ such that
the function
$$ N(t) = \sum_{r=0}^{2R} \kappa^{|r-R|} \tau^{-2r} M(\tau^{r} t), \quad t \geq 0, $$
satisfies
$$ \frac{1}{1+\eta} \tau^{2} \leq \frac{N(\lambda \tau)}{N(\lambda)} \leq (1+\eta) \tau^{2}, \quad 0 < \lambda \leq 1. $$
\end{lem}

\begin{proof}
We choose $ 1 \leq q < 2 < p < \infty $ such that $ \kappa \cdot \tau^{2-p} < 1 $ and $ \kappa \cdot \tau^{q-2} < 1 $. Since $ \alpha_{M} = \beta_{M} = 2 $, there are $ C, c > 0 $ such that
$$ c \Big( \frac{v}{u} \Big)^{p} \leq \frac{M(v)}{M(u)} \leq C \Big( \frac{v}{u} \Big)^{q}, \quad 0 < v \leq u \leq 1. $$
We pick a large enough $ R \in \mathbb{N} $ such that
$$ \kappa - (\kappa \cdot \tau^{2-p})^{R} \cdot \frac{\kappa}{c} > \frac{1}{1+\eta} \quad \textrm{and} \quad \frac{1}{\kappa} + (\kappa \cdot \tau^{q-2})^{R} \cdot C \tau^{q-2} < 1 + \eta. $$

Let us show that the choice of $ R $ works. Fix $ 0 < t \leq 1 $. We have
$$
\begin{aligned}
\tau^{-2} N(\tau t) - \kappa^{\pm 1} N(t) & = \sum_{r=0}^{2R} \kappa^{|r-R|} \tau^{-2(r+1)} M(\tau^{r+1} t) - \kappa^{\pm 1} \sum_{r=0}^{2R} \kappa^{|r-R|} \tau^{-2r} M(\tau^{r} t) \\
 & = \sum_{s=1}^{2R+1} \kappa^{|s-1-R|} \tau^{-2s} M(\tau^{s} t) - \sum_{s=0}^{2R} \kappa^{\pm 1+|s-R|} \tau^{-2s} M(\tau^{s} t) \\
 & = \kappa^{R} \tau^{-2(2R+1)} M(\tau^{2R+1} t) - \kappa^{\pm 1+R} M(t) \phantom{\sum_{s=1}^{2R}} \\
 & \quad \quad + \sum_{s=1}^{2R} (\kappa^{|s-1-R|} - \kappa^{\pm 1+|s-R|}) \tau^{-2s} M(\tau^{s} t),
\end{aligned}
$$
and since $ \kappa^{1+|s-R|} \leq \kappa^{|s-1-R|} \leq \kappa^{-1+|s-R|} $, we obtain
$$ \tau^{-2} N(\tau t) - \kappa N(t) \geq - \kappa^{1+R} M(t), \quad \tau^{-2} N(\tau t) - \frac{1}{\kappa} N(t) \leq \kappa^{R} \tau^{-2(2R+1)} M(\tau^{2R+1} t). $$
As $ N(t) = \sum_{r=0}^{2R} \kappa^{|r-R|} \tau^{-2r} M(\tau^{r} t) \geq \kappa^{0} \tau^{-2R} M(\tau^{R} t) $, we get
$$ M(t) = \frac{M(t)}{M(\tau^{R} t)} \cdot \tau^{2R} \cdot \tau^{-2R} M(\tau^{R} t) \leq \frac{1}{c} (\tau^{-R})^{p} \cdot \tau^{2R} \cdot N(t) = \frac{1}{c} \cdot \tau^{R(2-p)} \cdot N(t), $$
$$
\begin{aligned}
\tau^{-2(2R+1)} M(\tau^{2R+1} t) & = \frac{M(\tau^{2R+1} t)}{M(\tau^{R} t)} \cdot \tau^{-2(R+1)} \cdot \tau^{-2R} \cdot M(\tau^{R} t) \\
 & \leq C \cdot (\tau^{R+1})^{q} \cdot \tau^{-2(R+1)} \cdot N(t) = C \cdot \tau^{(R+1)(q-2)} \cdot N(t).
\end{aligned}
$$
Hence,
$$ \tau^{-2} N(\tau t) \geq \Big( \kappa - \kappa^{1+R} \cdot \frac{1}{c} \cdot \tau^{R(2-p)} \Big) \cdot N(t) > \frac{1}{1+\eta} \cdot N(t), $$
$$ \tau^{-2} N(\tau t) \leq \Big( \frac{1}{\kappa} + \kappa^{R} \cdot C \cdot \tau^{(R+1)(q-2)} \Big) \cdot N(t) < (1+\eta) \cdot N(t), $$
which completes the proof.
\end{proof}

\begin{lem} \label{lem:GenCaseStep2}
Let $ M $ be an Orlicz function such that $ \alpha_{M} = \beta_{M} = 2 $. Then for every $ \varepsilon > 0 $ and $ 0 < \nu < 1 $, there are $ \sigma \in \mathbb{N} $ and $ \lambda_{1}, \dots, \lambda_{\sigma} \in (0, 1] $ such that the function
$$ N(t) = \sum_{j=1}^{\sigma} M(\lambda_{j} t), \quad t \geq 0, $$
satisfies $ N(1) \geq 1 $ and
$$ \frac{1}{1+\varepsilon} t^{2} \leq \frac{N(\lambda t)}{N(\lambda)} \leq (1+\varepsilon) t^{2}, \quad 0 < \lambda \leq 1, \; \nu \leq t \leq 1. $$
\end{lem}

\begin{proof}
We pick a rational $ 0 < \tau < 1 $ such that $ \tau^{2} > \frac{1}{1+\varepsilon} $. Let $ L \in \mathbb{N} $ be large enough to satisfy $ \tau^{L} \leq \nu $. For a small enough $ \eta > 0 $, we have
$$ \frac{1}{(1+\eta)^{L}} \tau^{2} > \frac{1}{1+\varepsilon}. $$
Next, we pick a rational $ 0 < \kappa < 1 $ such that $ \kappa > \frac{1}{1+\eta} $. Let $ R \in \mathbb{N} $ be given by Lemma~\ref{lem:GenCaseStep1}. For a suitable $ K \in \mathbb{N} $, the function
$$ N(t) = K \cdot \sum_{r=0}^{2R} \kappa^{|r-R|} \tau^{-2r} M(\tau^{r} t), \quad t \geq 0, $$
satisfies $ N(1) \geq 1 $ and the numbers $ K \cdot \kappa^{|r-R|} \tau^{-2r} $, where $ 0 \leq r \leq 2R $, are natural.

Since the present $ N $ is a multiple of $ N $ from Lemma~\ref{lem:GenCaseStep1}, the property
$$ \frac{1}{1+\eta} \tau^{2} \leq \frac{N(\lambda \tau)}{N(\lambda)} \leq (1+\eta) \tau^{2}, \quad 0 < \lambda \leq 1, $$
is preserved. A simple induction argument gives
$$ \frac{1}{(1+\eta)^{l}} \tau^{2l} \leq \frac{N(\lambda \tau^{l})}{N(\lambda)} \leq (1+\eta)^{l} \tau^{2l}, \quad l \in \mathbb{N}, \; 0 < \lambda \leq 1. $$
For $ \nu \leq t \leq 1 $, as $ \tau^{L} \leq \nu $, there is $ l \in \{ 1, 2, \dots, L \} $ such that $ \tau^{l} \leq t \leq \tau^{l-1} $, thus
$$ \frac{N(\lambda t)}{N(\lambda)} \geq \frac{N(\lambda \tau^{l})}{N(\lambda)} \geq \frac{1}{(1+\eta)^{l}} \tau^{2l} \geq \frac{1}{(1+\eta)^{L}} \tau^{2} \tau^{2(l-1)} \geq \frac{1}{1+\varepsilon} t^{2}, $$
$$ \frac{N(\lambda t)}{N(\lambda)} \leq \frac{N(\lambda \tau^{l-1})}{N(\lambda)} \leq (1+\eta)^{l-1} \tau^{2(l-1)} \leq (1+\eta)^{L} \tau^{-2} \tau^{2l} \leq (1+\varepsilon) t^{2}. $$

It remains to realize that $ N $ is of the form $ N(t) = \sum_{j=1}^{\sigma} M(\lambda_{j} t) $ for some $ \sigma $ and $ \lambda_{1}, \dots, \lambda_{\sigma} $. We have $ N(t) = \sum_{r=0}^{2R} \omega_{r} M(\tau^{r} t) $ for the natural numbers $ \omega_{r} = K \cdot \kappa^{|r-R|} \tau^{-2r} $. So, we can consider a sequence $ \lambda_{1}, \dots, \lambda_{\sigma} $ consisting of the numbers $ \tau^{r}, 0 \leq r \leq 2R, $ in which $ \tau^{r} $ appears exactly $ \omega_{r} $ times.
\end{proof}

\begin{proof}[Proof of Theorem \ref{thm:main}]
For every $ k \in \mathbb{N} $, Lemma~\ref{lem:GenCaseStep2} provides us with numbers $ \sigma(k) \in \mathbb{N} $ and $ \lambda^{(k)}_{1}, \lambda^{(k)}_{2}, \dots, \lambda^{(k)}_{\sigma(k)} \in (0, 1] $ such that the function
$$ N_{k}(t) = \sum_{j=1}^{\sigma(k)} M(\lambda^{(k)}_{j} t), \quad t \geq 0, $$
satisfies $ N_{k}(1) \geq 1 $ and
$$ \frac{1}{1+2^{-k}} t^{2} \leq \frac{N_{k}(\lambda t)}{N_{k}(\lambda)} \leq (1+2^{-k}) t^{2}, \quad 0 < \lambda \leq 1, \; 2^{-k} \leq t \leq 1. $$
By Proposition~\ref{Prop:IsomOrlicz}, since $ M(\lambda^{(k)}_{1} t) \leq N_{k}(t) \leq \sigma(k) M(t) $, the spaces $ h_{M} $ and $ h_{N_{k}} $ are isomorphic, thus $ h_{N_{k}} $ is not Hilbertian. For a large enough $ s(k) \in \mathbb{N} $, the Banach-Mazur distance of the subspace $ \mathrm{span} \{ e_{1}, \dots, e_{s(k)} \} $ of $ h_{N_{k}} $ to the $ s(k) $-dimensional Euclidean space is at least $ k $.

Now, let $ M_{1}, M_{2}, \dots, $ be the sequence of the form $ N_{1}, \dots, N_{1}, N_{2}, \dots, N_{2}, N_{3}, \dots, N_{3}, \dots $, in which $ N_{k} $ appears $ s(k) $ times. Due to Lemma~\ref{lem:AsympHilb}, the space $ h_{(M_{i})_{i}} $ is $ 1^{+} $-asymptotically Hilbertian. At the same time, it is not Hilbertian. Indeed, for any $ k $, as $ N_{k} $ appears $ s(k) $ times in the sequence $ (M_{i})_{i} $, the space $ h_{(M_{i})_{i}} $ contains an isometric copy of $ \mathrm{span} \{ e_{1}, \dots, e_{s(k)} \} $ in $ h_{N_{k}} $, whose distance to $ \ell_{2}^{s(k)} $ is at least $ k $.

It remains to show that $ h_{M} $ contains an isometric copy of $ h_{(M_{i})_{i}} $. Let $ k_{i} $ be the numbers such that $ M_{i} = N_{k_{i}} $. We can find a sequence $ x_{1}, x_{2}, \dots $ of vectors in $ h_{M} $ with finite pairwise disjoint supports such that $ x_{i} $ is a rearrangement of the vector $ (\lambda^{(k_{i})}_{1}, \dots, \lambda^{(k_{i})}_{\sigma(k_{i})}, 0, 0, \dots) $. Then $ x_{1}, x_{2}, \dots $ is $ 1 $-equivalent with the canonical basis $ e_{i} $ in $ h_{(M_{i})_{i}} $.
\end{proof}

\section{Some ergodic twisted Hilbert spaces}\label{sec:Twisted}

We start with recalling some basic facts about twisted Hilbert spaces.

\begin{defin}
    A \textit{twisted Hilbert space} is a Banach space $X$ containing a Hilbertian subspace $Y$ such that the quotient $X/Y$ is Hilbertian.
\end{defin}

Hilbertian spaces are of course twisted Hilbert, but there are also non-Hilbertian twisted Hilbert spaces. The first example was built by Enflo, Lindenstrauss and Pisier \cite{EnfloLindenstraussPisier}, and shortly after, a general theory of twisted sums of quasi-Banach spaces (generalizing the notion of twisted Hilbert spaces) was developed by Kalton and Peck in \cite{KaltonPeck}, where a class of new examples of twisted Hilbert spaces, the spaces $\ell_2(\phi)$, were exhibited. Following \cite{KaltonCentralizersBook}, we recall below a method for constructing twisted Hilbert spaces based on centralizers, that can be seen as a particular case of the general construction of Kalton and Peck. In this section, when considering measures spaces $(E, \mu)$, it will always be assumed that $E$ is a Polish space and $\mu$ a sigma-finite Borel measure; this ensures in particular that spaces $L_p(\mu)$, $1 \leqslant p < \infty$, are separable. Recall that if $(E, \mu)$ is a measure space, $L_0(\mu)$ denotes the vector space of all measurable maps $E \to \mathbb{K}$ considered up to equality almost everywhere.

\begin{defin}
    Let $(E, \mu)$ be a measure space. A \textit{centralizer} on $L_2(\mu)$ is a homogeneous map $\Omega \colon L_2(\mu) \to L_0(\mu)$ for which there exists a constant $C \geqslant 0$ such that for every $f \in L_2(\mu)$ and $u \in L_\infty(\mu)$, one has $\Omega(uf) - u\Omega(f) \in L_2(\mu)$ and $$\|\Omega(uf) - u\Omega(f)\|_{L_2(\mu)} \leqslant C\|u\|_{L_\infty(\mu)}\|f\|_{L_2(\mu)}.$$
\end{defin}

\noindent (Maps $\Omega$ as above are called \textit{homogeneous centralizers} in \cite{KaltonCentralizersBook}, but here we decided to follow the more recent terminology from \cite{CabelloCastillo}.) Recall that a \textit{quasinorm} on a vector space $X$ is a map $\|\cdot\| \colon X \to \R_+$ satisfying the same axioms as a norm except the triangle inequality, which is replaced by the weaker inequality $\|x+y\| \leqslant C(\|x\| + \|y\|)$ for some fixed constant $C \geqslant 1$; equivalence of quasinorms is defined the same way as equivalence of norms. The theorem below summarizes the construction of twisted Hilbert spaces given in \cite[Section 4]{KaltonCentralizersBook}.

\begin{thm}\label{thm:TwistedFromCentralizer}
    Let $(E, \mu)$ be a measure space and $\Omega \colon L_2(\mu) \to L_0(\mu)$ be a centralizer. Let $$L_2(\mu) \oplus_\Omega L_2(\mu) \coloneq \{(f, g) \in L_0(\mu) \times L_2(\mu) \mid f - \Omega(g) \in L_2(\mu)\}.$$
    Then $L_2(\mu) \oplus_\Omega L_2(\mu)$ is a vector space and $$\|(f, g)\|_\Omega \coloneq \|f-\Omega(g)\|_{L_2(\mu)} + \|g\|_{L_2(\mu)}$$ defines a quasinorm on it. This quasinorm is equivalent to a complete norm, and the Banach space $L_2(\mu) \oplus_\Omega L_2(\mu)$ obtained in this way is a twisted Hilbert space.
\end{thm}

We now recall the construction of the spaces $\ell_2(\phi)$ of Kalton and Peck \cite{KaltonPeck}, following the approach of \cite{KaltonCentralizersBook} based on centralizers. Here, $\phi$ stands for any Lipschitz map $\R \to \R$. We will apply the theory of centralizers presented above in the special case when $E=\N$ and $\mu$ is the counting measure (although this construction is valid in more generality, see \cite[Section 3.12]{CabelloCastillo}). Let $\Omega_\phi \colon \ell_2 \to \mathbb{K}^\N$ be the unique homogeneous map such that, for every $x \in S_{\ell_2}$ and $k \in \N$, one has: $$\Omega_\phi(x)(k) = \left\{
\begin{array}{ll}
x(k)\phi(-\log|x(k)|) & \text{ if }x(k) \neq 0,\\
0 & \text{ if }x(k) = 0.
\end{array}
\right.$$
It follows from arguments from \cite{KaltonCentralizersBook} (see the computation at the top of page 14 and the proof of Proposition 4.1) that $\Omega_\phi$ is a centralizer on $\ell_2$. The associated twisted Hilbert space $\ell_2 \oplus_{\Omega_\phi} \ell_2$ is denoted by $\ell_2(\phi)$. For $\phi(t) = t$, this space is now known as the \textit{Kalton--Peck space} and often denoted by $Z_2$. It is easily seen that if $\phi$ is bounded on $\R^+$, then $\ell_2(\phi)$ is equal to the direct sum $\ell_2 \oplus \ell_2$ (up to equivalence of norms), so in particular, $\ell_2(\phi)$ is Hilbertian. The following theorem, which is the main result of this section, pictures a quite opposite situation in the case when $\phi$ is unbounded on $\R_+$.

\begin{thm}\label{ergodicityTwisted}
    Let $\phi \colon \R \to \R$ be a Lipschitz function. If $\phi$ is unbounded on $\R_+$, then the space $\ell_2(\phi)$ contains a non-Hilbertian asymptotically Hilbertian subspace. In particular, spaces $\ell_2(\phi)$ are either Hilbertian, or ergodic.
\end{thm}

Before proceeding to the proof of Theorem \ref{ergodicityTwisted}, let us make a few comments. A particular case of Question \ref{quest:ErgodicityTwisted} this theorem rises is the following.

\begin{question}\label{quest:ErgCentralizer}
Are all non-Hilbertian spaces of the form $L_2(\mu) \oplus_\Omega L_2(\mu)$, where $(E, \mu)$ is a measure space and $\Omega$ a centralizer on $L_2(\mu)$, ergodic?
\end{question}

The class of all complex spaces of the form $L_2(\mu) \oplus_\Omega L_2(\mu)$ as above, where 
the centralizer $\Omega$ is additionally supposed to be real (i.e., $\Omega(f)$ is real-valued whenever $f$ is), is of particular interest because, by a result of Kalton \cite[Theorem 7.6]{KaltonTwistedInterpolation}, 
this class coincides 
with the class of all twisted Hilbert spaces that can be generated by complex interpolation between two K\"othe function spaces.  
Let us explain this result in more details. \textit{K\"othe function spaces} on a measure space $(E, \mu)$ form a pretty general class of Banach spaces of measurable functions on $(E, \mu)$, including for instance all spaces $L_p(\mu)$ and more generally all Orlicz function spaces $L_M(\mu)$ (so, in particular, all Orlicz sequence spaces $\ell_M$ when $(E, \mu)$ is $\N$ endowed with the counting measure); see \cite[Section 3]{KaltonTwistedInterpolation} for a precise definition. Given two complex K\"othe function spaces $X_0$ and $X_1$ on the same measure space $(E, \mu)$, \textit{complex interpolation} is a construction that associates to each $\theta \in [0, 1]$ a K\"othe function space $[X_0, X_1]_\theta$ on $(E, \mu)$ that can be seen as a $\theta$-weighted average between $X_0$ and $X_1$ (for instance, $[L_p(\mu), L_q(\mu)]_\theta=L_r(\mu)$, where $1/r=(1-\theta)/p + \theta/q$); we refer the reader to \cite{BerghLofstrom} for a general introduction to this topic. The map $\theta \mapsto [X_0, X_1]_\theta$ is called an \textit{interpolation scale}. As observed by Kalton \cite{KaltonTwistedInterpolation} (formalizing earlier ideas by Rochberg and Weiss \cite{RochbergWeiss}), to each complex interpolation scale $[X_0, X_1]$ of K\"othe function spaces on a measure space $(E, \mu)$ satisfying $[X_0, X_1]_{1/2} = L_2(\mu)$, one can associate a twisted Hilbert space $[X_0, X_1]_{1/2}'$ which is in some sense the differential of the interpolation scale at $\theta = 1/2$. This space can be shown to have the form $L_2(\mu) \oplus_\Omega L_2(\mu)$, where $\Omega$ is a real centralizer on $L_2(\mu)$. The forementioned result by Kalton asserts that conversely, every twisted sum of the form $L_2(\mu) \oplus_\Omega L_2(\mu)$, where $\Omega$ is a real centralizer on the complex space $L_2(\mu)$, can be generated by complex interpolation between two K\"othe function spaces on $(E, \mu)$. (Kalton's result actually applies to more general twisted sums, but here we chose to restrict ourselves to twisted Hilbert spaces for simplicity of exposition.)

The case of interpolation between Orlicz sequence spaces has been studied in more detail by Castillo, Ferenczi and Gonzalez in \cite{CastilloFerencziGonzalez}. Consider two Orlicz functions $M_0$ and $M_1$, both satisfying the $\Delta_2$-condition at $0$ and such that both $M_0(t)/t$ and $M_1(t)/t$ tend to $0$ as $t$ tends to $0$. Moreover assume that $M_0^{-1}(t)M_1^{-1}(t) = t$ for all $t \geqslant 0$. It follows from \cite[Proposition 3.8]{CastilloFerencziGonzalez} that $[\ell_{M_0}, \ell_{M_1}]_{1/2} = \ell_2$. The twisted Hilbert space $[\ell_{M_0}, \ell_{M_1}]_{1/2}'$ is computed in \cite[Proposition 3.9]{CastilloFerencziGonzalez}; it is $\ell_2 \oplus_\Omega \ell_2$, where the centralizer $\Omega$ is given by:
$$\Omega(x)(k) = \left\{
\begin{array}{ll}
\displaystyle 2x(k)\log\frac{M_1^{-1}(|x(k)|^2)}{|x(k)|} & \text{ if }x(k) \neq 0,\\
0 & \text{ if }x(k) = 0,
\end{array}
\right.$$
for all $x \in S_{\ell_2}$. It is easily seen that $\Omega = \Omega_\phi$, where $\phi(t) = 2t + 2\log M_1^{-1}(e^{-2t})$. To see that $\phi$ is Lipschitz, it is enough to see that $\psi(t) = \log M_1^{-1}(e^{t})$ is Lipschitz; this follows from the fact that $\psi$ is nondecreasing and for every $s \in \R$ and $t \geqslant 0$, one has:
$$\psi(s+t) = \log M_1^{-1}(e^te^s) \leqslant \log \left(e^tM_1^{-1}(e^s)\right) = t + \psi(s),$$
the inequality coming from the concavity of $M_1^{-1}$. Hence $[\ell_{M_0}, \ell_{M_1}]_{1/2}'$ is isomorphic to $\ell_2(\phi)$, so we obtain the following corollary of Theorem \ref{ergodicityTwisted}.

\begin{cor}
    Every twisted Hilbert space generated by interpolation between two Orlicz sequences spaces as presented above is either Hilbertian or ergodic.
\end{cor}

We now turn to the proof of Theorem \ref{ergodicityTwisted}. For this we will find an isomorphic copy of an Orlicz sequence space in the space $\ell_2(\phi)$. Note that this has already been done by Kalton and Peck for some classes of functions $\phi$, see \cite[Lemma 5.3]{KaltonPeck}. Below we generalize their argument to all Lipschitz functions $\phi$.

\begin{prop}\label{TwistedContainsOrlicz}
    Let $\phi \colon \R \to \R$ be a Lipschitz function. For every $n \in \N$, let $w_n \coloneq (\Omega_\phi(e_n), e_n) \in \ell_2(\phi)$, where $e_n$ denotes the $n$-th basic vector in $\ell_2$. Then $(w_n)_{n\in \N}$ is equivalent to the canonical basis of an Orlicz space $h_M$ with $\alpha_M = \beta_M = 2$. Moreover, if $\phi$ is unbounded on $\R_+$, then $h_M$ is non-Hilbertian.
\end{prop}

Theorem \ref{ergodicityTwisted} will be an immediate consequence of Proposition \ref{TwistedContainsOrlicz} and Theorems \ref{thm:main} and \ref{thm:Anisca}. Below, we prove Proposition \ref{TwistedContainsOrlicz}. We start with a preliminary lemma.

\begin{lem} \label{SuitableOrliczFunction}
    Let $\phi \colon \R \to \R$ be a Lipschitz function. Then there exists an Orlicz function $M$ and constants $c, C > 0$ such that for every $t > 0$,
    \begin{equation}\label{IneqLipschitzOrlicz}cM(t) \leqslant t^2(1+ \phi(-\log t)^2) \leqslant CM(t).\end{equation}
\end{lem}

\begin{proof}
    One can assume that $\phi \geqslant 0$. Let $L$ denote the Lipschitz constant of $\phi$. Let $f \colon \R \to \R_+$ be a $\mathcal{C}^\infty$ map supported on $[0, 1]$ having integral $1$. We let $\psi \coloneq \phi * f$, that is,
    $$\psi(x) \coloneq \int_\R \phi(y) f(x-y) dy.$$

    \begin{claim}\label{claim:convolution}
        \alaligne
        \begin{enumerate}[(1)]
            \item The function $\psi - \phi$ is bounded on $\R$.
            \item The function $\psi$ is $\mathcal{C}^\infty$, and for every $n \geqslant 1$, its $n$-th derivative $\psi^{(n)}$ is bounded on $\R$.
        \end{enumerate}
    \end{claim}

    \begin{proof}
        We prove both points simultaneously. The function $\psi$ is clearly $\mathcal{C^\infty}$ with $\psi^{(n)} = \phi * f^{(n)}$. Let $n \geqslant 0$ and $x \in \R$. One has:
        $$\psi^{(n)}(x) - \phi(x)\int_\R f^{(n)}(x-y) dy = \int_{x-1}^x (\phi(y) - \phi(x))f^{(n)}(x - y) dy,$$
        and $|\phi(y) - \phi(x)| \leqslant L$ for $y \in [x-1, x]$, so
        $$\left|\psi^{(n)}(x) - \phi(x)\int_\R f^{(n)}(x-y) dy\right| \leqslant L\|f^{(n)}\|_\infty.$$
        Observing that $\int_{\R}f(x-y)dy = 1$, we get (1). Observing that $\int_{\R}f^{(n)}(x-y)dy = 0$ for $n \geqslant 1$, we get (2).
    \end{proof}

    We now let $K \geqslant 1$ be a constant that will be fixed later and we let, for $t > 0$:
    $$M(t) \coloneq t^2(K+\psi(-\log t)^2),$$
    and $M(0) = 0$, so that $M$ is continuous on $\R_+$. We also clearly have $\lim_{t \to +\infty} M(t) =  +\infty$. For $t > 0$, letting $x \coloneq - \log t$, a straightforward computation shows that
    $$M''(t) = 2(\psi(x)^2 + \psi(x)(\psi''(x) - 3\psi'(x)) + \psi'(x)^2 + K).$$
    Letting $b \coloneq \inf_\R \psi'' - 3\psi'$, finite by Claim \ref{claim:convolution}, and remembering that $\phi(x) \geqslant 0$, we get:
    $$M''(t) \geqslant 2(\psi(x)^2 + b\psi(x) + K) \geqslant \frac{4K - b^2}{2},$$
    which is nonnegative as soon as $4K \geqslant b^2$. So, for $K \geqslant 1$ chosen large enough, the function $M$ is Orlicz.

    It now remains to show that Inequality (\ref{IneqLipschitzOrlicz}) holds for some $c, C > 0$. First, we note that
    \begin{equation} \label{IneqAbstract}
    \frac{1+u^{2}}{1+v^{2}} \leq 2 \big( 1+(u-v)^{2} \big), \quad u, v \in \mathbb{R}.
    \end{equation}
    Indeed, if we denote $ \Delta = u - v $, then $ 1 + u^{2} = 1 + (v + \Delta)^{2} = 1 + v^{2} + 2v\Delta + \Delta^{2} \leq 1 + v^{2} + v^{2} + \Delta^{2} + \Delta^{2} \leq 2(1 + v^{2})(1 + \Delta^{2}) $.

    Finally, let $ A \coloneq \| \psi - \phi \|_{\infty} $, finite by Claim \ref{claim:convolution}. Using \eqref{IneqAbstract}, we deduce that
    $$ \frac{1}{2(1+A^2)} \leq \frac{1+\psi(-\log t)^2}{1+ \phi(-\log t)^2} \leq \frac{M(t)}{t^2(1+ \phi(-\log t)^2)} \leq K \cdot \frac{1+\psi(-\log t)^2}{1+ \phi(-\log t)^2} \leq K \cdot 2(1+A^2) $$
    for each $ t > 0 $.
\end{proof}

\begin{proof}[Proof of Proposition~\ref{TwistedContainsOrlicz}]
Let $ M $ be an Orlicz function given by Lemma~\ref{SuitableOrliczFunction}. We denote by $ L $ the Lipschitz constant of $ \phi $. Due to \eqref{IneqAbstract}, for $ \lambda, t \in (0, 1] $ and $ 1 < q < \infty $, we have
\begin{align*}
\Big( \frac{M(\lambda t)}{M(\lambda)t^{q}} \Big)^{\pm 1} & \leq \frac{C}{c} \Big( \frac{\lambda^{2}t^{2}(1+\phi(-\log(\lambda t))^{2})}{t^{q}\lambda^{2}(1+\phi(-\log \lambda)^{2})} \Big)^{\pm 1} \\
 & \leq \frac{C}{c} t^{\pm (2-q)} \cdot 2 \big( 1 + (\phi(-\log(\lambda t)) - \phi(-\log \lambda))^{2} \big) \\
 & \leq \frac{C}{c} t^{\pm (2-q)} \cdot 2 \big( 1 + L^{2} (\log t)^{2} \big).
\end{align*}
The number $ K_{s} = \sup_{t \in (0, 1]} \frac{C}{c} t^{s} \cdot 2(1 + L^{2} (\log t)^{2}) $ is finite for every $ s > 0 $. For $ q \in [1, 2) $, we have $ \frac{M(\lambda t)}{M(\lambda)t^{q}} \leq K_{2-q} $, and for $ q \in (2, \infty) $, we have $ \frac{M(\lambda t)}{M(\lambda)t^{q}} \geq 1/K_{q-2} $. Therefore, $ \alpha_{M} = \beta_{M} = 2 $. This in turn implies that $M$ satisfies the $\Delta_2$-condition at $0$; in particular, a vector $x \in \mathbb{K}^\N$ belongs to $h_M$ iff $\sum_{n=1}^\infty M(|x(n)|) < \infty$.

To show that $ w_{1}, w_{2}, \dots $ is equivalent to the canonical basis of $ h_{M} $, we need to check that for any sequence $ t_{n} $, the sum $ \sum_{n=1}^{\infty} t_{n} w_{n} $ converges in $ \ell_{2}(\phi) $ if and only if the sum $ \sum_{n=1}^{\infty} t_{n} e_{n} $ converges in $ h_{M} $. We can assume that $ t_{n} \neq 0 $ for each $ n $. First, we realize that we can also assume that
$$ \sigma \coloneq \Big( \sum_{n=1}^{\infty} |t_{n}|^{2} \Big)^{1/2} < \infty. $$Indeed, this is implied by both conditions:
\begin{itemize}
    \item it is clear if $ \sum_{n=1}^{\infty} t_{n} w_{n} $ converges, since $ \Vert \sum_{n=1}^{l} t_{n} e_{n} \Vert_{\ell_{2}} \leq \Vert \sum_{n=1}^{l} t_{n} w_{n} \Vert_{\Omega_{\phi}} $ for $ l \geq 1 $;
    \item on the other hand, assuming that $ \sum_{n=1}^{\infty} t_{n} e_{n} $ converges in $ h_{M} $, we can write $ \sum_{n=1}^{\infty} |t_{n}|^{2} \leq \sum_{n=1}^{\infty} |t_{n}|^{2} (1+\phi(-\log |t_{n}|)^{2}) \leq C \sum_{n=1}^{\infty} M(|t_{n}|) < \infty $.
\end{itemize}
\noindent Next, we show that the convergence of $ \sum_{n=1}^{\infty} t_{n} w_{n} $ is equivalent to
\begin{equation} \label{ConvergCond1}
\sigma + \Big( \sum_{n=1}^{\infty} |t_{n}|^{2} \big[ \phi(-\log(|t_{n}|/\sigma)) - \phi(0) \big]^{2} \Big)^{1/2} < \infty.
\end{equation}
Although this basically follows from the proof of \cite[Lemma~5.3]{KaltonPeck}, where it is shown in a more general situation (the expression above is the expression $ (**) $ from \cite{KaltonPeck}), for the convenience of the reader, we provide some details. Note first that $ w_{n} = (\phi(0)e_{n}, e_{n}) $ for $ n \in \mathbb{N} $. For $ N \in \mathbb{N} $, we denote
$$ \sigma_{N} \coloneq \Big( \sum_{n=1}^{N} |t_{n}| \Big) ^{1/2}. $$
We have
\begin{align*}
\Big\Vert \sum_{n=1}^{N} t_{n} w_{n} \Big\Vert_{\Omega_{\phi}} & = \Big\Vert \Big( \phi(0) \sum_{n=1}^{N} t_{n} e_{n}, \sum_{n=1}^{N} t_{n} e_{n} \Big) \Big\Vert_{\Omega_{\phi}} \\
 & = \Big\Vert \phi(0) \sum_{n=1}^{N} t_{n} e_{n} - \Omega_{\phi} \Big( \sum_{n=1}^{N} t_{n} e_{n} \Big) \Big\Vert_{\ell_{2}} + \Big\Vert \sum_{n=1}^{N} t_{n} e_{n} \Big\Vert_{\ell_{2}} \\
 & = \Big\Vert \phi(0) \sum_{n=1}^{N} t_{n} e_{n} - \sigma_{N} \Omega_{\phi} \Big( \frac{1}{\sigma_{N}} \sum_{n=1}^{N} t_{n} e_{n} \Big) \Big\Vert_{\ell_{2}} + \sigma_{N} \\
 & = \Big\Vert \phi(0) \sum_{n=1}^{N} t_{n} e_{n} - \sigma_{N} \sum_{n=1}^{N} \frac{t_{n}}{\sigma_{N}} \phi \Big( -\log \Big| \frac{t_{n}}{\sigma_{N}} \Big| \Big) e_{n} \Big\Vert_{\ell_{2}} + \sigma_{N} \\
 & = \Big\Vert \sum_{n=1}^{N} t_{n} \Big[ \phi(0) - \phi \Big( -\log \Big| \frac{t_{n}}{\sigma_{N}} \Big| \Big) \Big] e_{n} \Big\Vert_{\ell_{2}} + \sigma_{N},
\end{align*}
and thus
$$ \bigg| \Big\Vert \sum_{n=1}^{N} t_{n} w_{n} \Big\Vert_{\Omega_{\phi}} - \bigg( \Big\Vert \sum_{n=1}^{N} t_{n} \Big[ \phi(0) - \phi \Big( -\log \Big| \frac{t_{n}}{\sigma} \Big| \Big) \Big] e_{n} \Big\Vert_{\ell_{2}} + \sigma_{N} \bigg) \bigg| $$
$$ \leq \Big\Vert \sum_{n=1}^{N} t_{n} \Big[ \phi \Big( -\log \Big| \frac{t_{n}}{\sigma} \Big| \Big) - \phi \Big( -\log \Big| \frac{t_{n}}{\sigma_{N}} \Big| \Big) \Big] e_{n} \Big\Vert_{\ell_{2}} \leq L\sigma_{N} \log \frac{\sigma}{\sigma_{N}}. $$
As $ L\sigma_{N} \log \frac{\sigma}{\sigma_{N}} \leq L\sigma \log \frac{\sigma}{\sigma_{1}} $, it follows that the sequence of partial sums of $ \sum_{n=1}^{\infty} t_{n} w_{n} $ is bounded if and only if \eqref{ConvergCond1} holds. It remains to recall that $ w_{n} $ is a boundedly complete basic sequence, see \cite{KaltonPeck}.

On the other hand, $ \sum_{n=1}^{\infty} t_{n} e_{n} $ converges in $ h_{M} $ if and only if $ \sum_{n=1}^{\infty} M(|t_{n}|) < \infty $, which is, by \eqref{IneqLipschitzOrlicz}, equivalent to
\begin{equation} \label{ConvergCond2}
\sigma^{2} + \sum_{n=1}^{\infty} |t_{n}|^{2} \phi(-\log|t_{n}|)^{2} < \infty.
\end{equation}
Now, comparing \eqref{ConvergCond1} and \eqref{ConvergCond2}, we see that what we need to check is that the sequence $ t_{n} [\phi(-\log(|t_{n}|/\sigma)) - \phi(0)] $ belongs to $ \ell_{2} $ if and only if the sequence $ t_{n} \phi(-\log|t_{n}|) $ belongs to $ \ell_{2} $. It is sufficient that the difference of these sequences belongs to $ \ell_{2} $. As
\begin{align*}
\big| t_{n} \big[ \phi(-\log(|t_{n}|/\sigma)) - \phi(0) \big] - t_{n} \phi(-\log |t_{n}|) \big| & \leq |t_{n}| \cdot \big( |\phi(0)| + L \cdot \big| -\log(|t_{n}|/\sigma) + \log |t_{n}| \big| \big) \\
 & = |t_{n}| \cdot \big( |\phi(0)| + L \cdot |\log \sigma| \big),
\end{align*}
the difference has $\ell_2$-norm at most $ \sigma \cdot (|\phi(0)| + L \cdot |\log \sigma|) $.

If $ \phi $ is unbounded on $ \mathbb{R}_{+} $, then it follows from \eqref{IneqLipschitzOrlicz} that the function $ M(t)/t^{2} $ is unbounded on $ (0, 1) $, so Proposition~\ref{Prop:IsomOrlicz} shows that the canonical basis of $h_M$ is not equivalent to the canonical basis of $\ell_2$. Since the canonical basis of $h_M$ is unconditional and all unconditional bases of $\ell_2$ are equivalent (see \cite[page 71]{LindenstraussTzafririI}), we deduce that $h_M$ is not Hilbertian.
\end{proof}

\paragraph{Acknowledgments.} The authors would like to thank Wilson Cuellar Carrera and Valentin Ferenczi for enlightening discussions, and the anonymous referee for useful comments.

\bibliographystyle{plain}
\bibliography{main}

  \par
  \bigskip
    \textsc{\footnotesize Univ. Lille, CNRS, UMR 8524 - Laboratoire Paul Painlevé, F-59000 Lille, France}
  
  \textit{E-mail address}: \texttt{nderancour@univ-lille.fr}
  
  \par 
  \bigskip
  \textsc{\footnotesize 
Institute of Mathematics,
Czech Academy of Sciences,
\v{Z}itn\'a 25,
115 67 Praha 1,
Czech Republic}

\textit{E-mail address}: \texttt{kurka.ondrej@seznam.cz}

\end{document}